\documentclass[a4paper,11pt]{article}
\usepackage{mathrsfs}
\usepackage{latexsym}
\usepackage{amsmath,amssymb}
\usepackage{amsthm}
\usepackage{amsfonts}
\usepackage[usenames]{color}
\usepackage{amssymb}
\usepackage{graphicx}
\usepackage{amsmath}
\usepackage{amsfonts}
\usepackage{amsthm}
\usepackage{mathrsfs}
\usepackage{dsfont}
\usepackage{indentfirst}
\usepackage{xcolor}

\def\phi{\varphi}

\newtheoremstyle{mythm}{1.5ex plus 1ex minus .2ex}{1.5ex plus 1ex
minus .2ex}{\kai}{\parindent}{\song\bfseries}{}{1em}{}
\numberwithin{equation}{section}
\newtheorem{definition}{Definition}[section]
\newtheorem{theorem}{Theorem}[section]
\newtheorem{lemma}{Lemma}[section]
\newtheorem{example}{Example}[section]
\newtheorem{proposition}{Proposition}[section]

\allowdisplaybreaks[4]

\begin{document}
\title{A sharp curvature lower bound for the first exterior p-harmonic Steklov eigenvalue}
\author{Yating Niu\footnote{Department of Mathematics, Nanjing Forestry University, Nanjing 210037, China. E-mail: ytniu@njfu.edu.cn}  and Tao Wang\footnote{Beijing International Center for Mathematical Research, Peking University, Beijing 100871, China. E-mail: taowang25@pku.edu.cn}}
\date{}
\maketitle

\begin{abstract}
In this paper, we study the first variational Steklov eigenvalue of the $ p$-Laplace equation on exterior domain $ \Omega^{\text{ext}}$ for $ 1< p <n$. If $ \Omega$ is convex and $ \partial \Omega\in C^{1,1}$, we prove a sharp lower bound in terms of $p$-logarithmic mean of the principal curvatures of $ \partial \Omega$. For the linear case $ p=2$, our estimate reduces to the logarithmic mean bound of Bundrock et al. \cite{LADMI2026}.  We also derive an upper bound in terms of a boundary isocapacitary constant. The analysis relies on the established finite energy theory on exterior domains, together with a decay estimate for $p$-harmonic extensions. 
\end{abstract}



\section{Introduction}

The Steklov eigenvalue problem arises from the Dirichlet--to--Neumann operator \cite{Steklov1902}. For a bounded Lipschitz domain $\Omega \subset \mathbb{R}^n$, the Steklov eigenvalue problem in $\Omega$ is given by 
\begin{equation*}
 \begin{cases}
  \Delta u(x) = 0, \quad &x \in \Omega, \\
  \frac{\partial u}{\partial \bar{\nu}}(x) = \sigma u(x), &x \in \partial \Omega,
 \end{cases}
\end{equation*}
where $\bar{\nu}$ is the unit normal pointing towards the exterior of $\Omega$. This problem has been extensively investigated on compact Riemannian manifolds; we refer the reader to the comprehensive surveys \cite{CGGS2024, GP2017} and the references therein for a more detailed account. More recently, Xiong \cite{Xiong2023} and Bundrock et al. \cite{LADMI2026} studied the exterior Steklov problem on Euclidean spaces, establishing results on nodal domains as well as upper and lower bounds for the first eigenvalue.

The nonlinear generalization of the Laplacian to the $p$-Laplacian, defined by $\Delta_p u = \mathrm{div}(|\nabla u|^{p-2}\nabla u)$ with $p > 1$, is fundamental in the study of non-Newtonian fluids and in geometric analysis. The Steklov eigenvalue problem can be analogously formulated for the $p$-Laplacian as well. For recent developments concerning the $p$-Laplacian Steklov problem on bounded Lipschitz domains, we refer to\cite{Le2006, MartinezRossi2002, Provenzano2022, Torne2005, Verma2020, WangWang2026}.

Inspired by the works \cite{LADMI2026, Han2016}, we investigate in this paper the exterior Steklov problem for the $p$-Laplacian on a Euclidean domain. Let $\Omega \subset \mathbb{R}^n$ be a bounded Lipschitz domain. Denote by
\[
 \Omega^{\text{ext}}:= \mathbb{R}^n \setminus \overline{\Omega}
\]
the exterior domain of $\Omega$. Throughout this paper, we assume that $\Omega^{\text{ext}}$ is connected. We first consider the following $p$-Laplace equation 
\begin{equation}\label{f1}
\begin{cases}
-\Delta_p u = 0 \quad & x \in \Omega^{\text{ext}}, \\
u(x) = g \qquad &  x \in \partial\Omega,
\end{cases}
\end{equation}
where $ g\in W^{1-\frac{1}{p}, p}(\partial\Omega)$. See Section~\ref{sec:preliminaries} for the definition of fractional Sobolev spaces.

Auchmuty and Han \cite{AH2014} introduced the space $E^{1, p}(\Omega^{\text{ext}})$ and proved that equation \eqref{f1} admits a unique solution in $E^{1, p}(\Omega^{\text{ext}})$; the definition of this space is given in Section~\ref{sec:preliminaries}. Indeed, if $u \in E^{1, p}(\Omega^{\text{ext}})$ is a solution of equation \eqref{f1}, then $u = O(|x|^{-\frac{n-p}{p-1}})$ as $|x| \to \infty$; see Proposition~\ref{prop:decay-estimate}. Consequently, for any $1 < p< n$, we may consider the following exterior Steklov problem for the $p$-Laplacian on a Euclidean domain.

\begin{definition}
The exterior Steklov problem on $ \Omega^{\text{ext}} \subset \mathbb{R}^n$ is defined as follows: find $ \sigma \in \mathbb{R}$ for which there exists a nonzero $ u$: $ \Omega^{\text{ext}} \to \mathbb{R}$ such that $ u \in W^{1,p}_{\text{loc}}(\Omega^{\text{ext}})$, and 
\begin{equation}\label{f9}
\begin{cases}
-\Delta_p u = 0 \quad &\text{ in } \ \Omega^{\text{ext}}, \\
|\nabla u|^{p-2} \frac{\partial u}{\partial \nu} = \sigma |u|^{p-2} u  &\text{ on } \  \partial\Omega, \\
u(x)= O(|x|^{-\frac{n-p}{p-1}}) \ \text{ as } \ |x| \to \infty,
\end{cases}\tag{ES}
\end{equation}
where $\nu = -\bar{\nu}$ is the unit normal pointing towards the exterior of $\Omega^{\text{ext}}$.  
\end{definition}

We show that any solution $u$ to \eqref{f9} actually possesses finite energy and consequently belongs to the space $E^{1, p}(\Omega^{\text{ext}})$; see Proposition~\ref{prop:exactly-space}. We then turn to the first variational eigenvalue, namely
\[
 \sigma_{1,p}(\Omega^{\text{ext}}) = \inf_{\substack{u \in E^{1,p}(\Omega^{\text{ext}}) \\ u \neq 0}} \frac{\displaystyle\int_{\Omega^{\text{ext}}} |\nabla u|^p \, \mathrm{d}x}{\displaystyle\int_{\partial \Omega} |u|^p \, \mathrm{d}S}.
\]
For any $1 < p < n$, we have $\sigma_{1, p}(\Omega^{\text{ext}}) > 0$, see \cite{Han2016}. We show that $\sigma_{1, p}(\Omega^{\text{ext}})$ is bounded from below by the minimum of the logarithmic mean of the principal curvatures.

Let $ k\in \mathbb{N}$, and $ \alpha_1, \cdots, \alpha_k$ be nonnegative real numbers. Following \cite{nvariable}, we define their  $p$-logarithmic mean $\mathcal{L}_p$ in $ k$ variables with respect to the Dirichlet measure as
\[
 \mathcal{L}_p(\alpha_1, \cdots, \alpha_k):= \left[\frac{\Gamma\left(\frac{k}{p-1}\right)}{\Gamma\left(\frac{1}{p-1}\right)^{k}}\int_{S_{k-1}}\frac{\prod_{j=1}^{k}u_j^{\frac{1}{p-1}-1}}{\sum_{j=1}^{k}\alpha_j u_j} \, \mathrm{d}u_1\cdots \mathrm{d}u_{k-1}\right]^{-1},
\]
where $S_{k-1} = \{u_j \geq 0: \sum_{j=1}^{k}u_j = 1\}$ is the $(k-1)$-dimensional standard simplex. The Dirichlet measure serves as an important example of a probability measure on $ S_{k-1}$. In fact, if precisely $q$ of the constants $ \alpha_1, \cdots, \alpha_k$ are positive and the remaining $ k-q$ are zero, then the integral over $ S_{k-1}$ diverges to infinity when $q \leq p-1$, in which case we set $ \mathcal{L}_p =0$; when $ q > p-1$, the integral is finite, and consequently $ \mathcal{L}_p >0$.
Note that when $p = 2$ and $\alpha_1, \cdots, \alpha_k$ are distinct positive real numbers, $\mathcal{L}_p(\alpha_1, \cdots, \alpha_k)$ coincides with the quantity $L(\alpha_1, \cdots, \alpha_k)$ defined in \cite{LADMI2026}. 

For $ j= 1, \cdots, n-1$ and $ s\in \partial \Omega$, let $ \kappa_j(s)$ denote the $ j$-th principal curvature of $ \partial \Omega$, where defined. We adopt the convention that the principal curvatures are nonnegative if $ \Omega$ is convex. We generalize the logarithmic-mean curvature bound established by Bundrock et al. \cite{LADMI2026} for the linear case $p = 2$ to the nonlinear setting $1 < p <n$.

\begin{theorem}\label{f8}
Suppose that $ \Omega \subset \mathbb{R}^n$ is a bounded convex domain with $ \partial \Omega \in C^{1,1}.$ Then for any $1 < p <n$,
\begin{equation}\label{equ:lower_bound}
\sigma_{1,p}(\Omega^{\text{ext}}) \geq \inf_{s\in \partial \Omega} \left[\frac{n-p}{p-1}\mathcal{L}_p(\kappa_1(s), \cdots, \kappa_{n-1}(s))\right]^{p-1}.
\end{equation}
\end{theorem}

For $\partial \Omega \in C^{1, 1}$, the principal curvatures are defined almost everywhere. Accordingly, the infimum in \eqref{equ:lower_bound} is taken over points where the curvatures are defined. If $\Omega$ is a ball, equality in \eqref{equ:lower_bound} holds.

We further investigate the relation between $\sigma_{1, p}$ and the isocapacitary constants. Using capacity to estimate eigenvalues can be traced back to the works of Maz'ya \cite{M1964, M2009, M1962}. Maz'ya introduced the isocapacitary constants and obtained upper and lower bounds for the first Dirichlet eigenvalue of the Laplacian on bounded domains. For the first Dirichlet eigenvalue of the $p$-Laplacian on bounded domains, Maz'ya type estimates also hold; see \cite{Gri1999}. We also mention \cite{HS2025, WangWang2026} for estimates of the Neumann and the Steklov eigenvalues of the $p$-Laplacian on bounded domains using the isocapacitary constants. 

For any nonempty set $A \subset \partial \Omega$ , define the class of admissible functions $\mathcal{S}(A)$ as 
\[
 \mathcal{S}(A):= \{u \in C^1(\Omega^{\text{ext}}) \cap C(\overline{\Omega^{\text{ext}}}): \mathrm{supp}\,(u) \text{ is bounded and } u|_A = 1\}.
\]
Consider the capacity of $A$ relative to $\partial \Omega$, that is, 
\[
 \mathrm{Cap}_p(A, \partial \Omega):= \inf_{u \in \mathcal{S}(A)}\int_{\Omega^{\text{ext}}}|\nabla u|^p \, \mathrm{d}x.
\]
Denote the area of $A$ by
\[
 \mathrm{Area}(A) = \int_A1 \, \mathrm{d}S,
\]
where $\mathrm{d}S$ is the $(n-1)$-dimensional Hausdorff measure on $\partial \Omega$. For any $1< p< n$, we define the isocapacitary constant $\Lambda_p(\partial \Omega)$ via
\[
 \Lambda_p(\partial \Omega):= \inf_{A \subset \partial \Omega}\frac{\mathrm{Cap}_p(A, \partial \Omega)}{\mathrm{Area}(A)}.
\]
Then we have the following theorem.
\begin{theorem}\label{thm:main_upper}
 Suppose that $\Omega \subset \mathbb{R}^n$ is a bounded and Lipschitz domain. Then for any $1 < p <n$,
 \[
  \sigma_{1, p}(\Omega^{\text{ext}}) \leq \Lambda_p(\partial \Omega).
 \]
\end{theorem}

It should be noted that both the lower bound in Theorem~\ref{f8} and the upper bound in Theorem~\ref{thm:main_upper} are sharp; see Example~\ref{exm:sharp}.


This paper is organized as follows. In Section \ref{sec:preliminaries}, we introduce some concepts and notations. In Section \ref{sec:eigenvalue}, we establish upper and lower bounds for the eigenvalues, and recover the explicit eigenvalue $ \sigma_{1,p}( B_R^{\text{ext}})$. This section also contains the proofs of Theorem \ref{f8} and Theorem \ref{thm:main_upper}.

\section{Preliminaries}\label{sec:preliminaries}
In this section, we introduce the basic notions and prove the comparison principle for the exterior domain. 

\begin{definition}
We say that $ u \in W^{1,p}_{\text{loc}}(\Omega^{\text{ext}})$ is a weak solution of problem \eqref{f1}, if $ u=g$ on $\partial \Omega$ in the sense of traces and 
\begin{equation}\label{f3}
\int_{\Omega^{\text{ext}}} |\nabla u|^{p-2} \nabla u \cdot \nabla \phi \, \mathrm{d}x =0
\end{equation}
for all $ \phi \in C_c^{\infty}(\Omega^{\text{ext}})$.
\end{definition}

For any domain $U \subset \mathbb{R}^n$, We recall some results related to the fractional Sobolev space $W^{s,p}(U).$ Let $ s\in (0,1)$, $ 1< p < \infty$ and set
\[
W^{s,p}(U) = \left\{u\in L^p(U) \ | \ [u]_{s,p} < \infty\right\}, 
\]
where the Gagliardo seminorm $ [u]_{s,p}$ is defined as
\[
[u]_{s,p} = \left(\iint_{U \times U} \frac{|u(x) - u(y)|^p}{|x-y|^{n+sp}} \, \mathrm{d}x\mathrm{d}y\right)^{\frac{1}{p}}.
\]
Equipped with the norm
\[
\| u \|_{W^{s,p}(U)} = \left( \int_{U} |u|^p \, \mathrm{d}x + \iint_{U \times U} \frac{|u(x) - u(y)|^p}{|x-y|^{n+sp}} \, \mathrm{d}x\mathrm{d}y\right)^{\frac{1}{p}},
\]
$ W^{s,p}(U)$ is a uniformly convex Banach space, and hence reflexive. A more detailed account of the properties of $ W^{s,p}(U)$ can be found in \cite{12Hitch}. For a bounded open set $ \Omega $, Gagliardo \cite{Gag57} proved that the trace operator $\mathrm{Tr}$ is linear, continuous and surjective from $ W^{1,p}(\Omega)$ onto $ W^{1-\frac{1}{p}, p}(\partial \Omega)$.


For any $1 < p <n$, we then recall the space $E^{1, p}$ introduced by Auchmuty and Han \cite{AH2014}. Although they initially consider $n \geq 3$, they note that the space is also well-defined for $n = 2$. Let $\Omega \subset \mathbb{R}^n$ be a bounded Lipschitz domain with connected $\Omega^{\text{ext}}$. We consider Lebesgue measurable functions $ u: \Omega^{\text{ext}} \to \mathbb{R}$ satisfying the following three conditions:
\begin{itemize}
    \item[(i)] $u \in L^1_{\text{loc}}(\Omega^{\text{ext}})$, 
    \item[(ii)] $\nabla u \in L^p(\Omega^{\text{ext}})$,
    \item[(iii)] $ \{x\in \Omega^{\text{ext}} : |u(x)| \geq C \}$ has finite Lebesgue measure for any $ C>0$. 
\end{itemize}
The set of all such Lebesgue measurable functions that satisfy (i)-(iii) be denoted $ E^{1,p} (\Omega^{\text{ext}})$. The functions in $ E^{1,p} (\Omega^{\text{ext}})$ are said to have finite $p$-energy on $\Omega^{\text{ext}}$. $E^{1, p}(\Omega^{\text{ext}})$ is a real Banach space with respect to the gradient $L^p$-norm
\[
\|\nabla u\|_{L^p(\Omega^{\text{ext}})} = \left( \int_{\Omega^{\text{ext}}} |\nabla u|^p \, \mathrm{d}x \right)^{\frac{1}{p}}. 
\]
Note that $W^{1,p}(\Omega^{\text{ext}}) \subset E^{1,p}(\Omega^{\text{ext}})\subset L^{p^*}(\Omega^{\text{ext}})$ with $ p^* = \frac{np}{n-p}$ (see \cite[Section 8.3]{Lieb01} and \cite[Corollary 3.3]{AH2014}).

In the following, we prove the decay estimate for solutions of equation \eqref{f1}.
\begin{proposition}\label{prop:decay-estimate}
 Suppose that $u \in E^{1, p}(\Omega^{\text{ext}})$ is a solution of equation \eqref{f1}, then $u = O(|x|^{-\frac{n-p}{p-1}})$ as $|x| \to \infty$.
\end{proposition}
\begin{proof}
We first claim that 
$$ u(x) \rightarrow 0 \quad \text{as} \quad  |x| \rightarrow + \infty.$$

From the embedding, $E^{1,p}(\Omega^{\text{ext}})\subset L^{p^*}(\Omega^{\text{ext}})$, it follows that $u \in L^{p^*}(\Omega^{\text{ext}})$. By local boundedness estimates for $ p$-harmonic functions (\cite[Lemma 3.6]{Notes}), for large $ R$, one obtains
\[
\operatorname{ess\,sup}_{B_{2R} \setminus B_R} |u(x)| \leq C(n,p)R^{-\frac{n}{p^*}} \|u\|_{L^{p^*}(B_{4R} \setminus B_{R/2})}. 
\]
Since $ u \in L^{p^*}(\Omega^{\text{ext}})$, taking the limit $ R \rightarrow \infty$ on the right-hand side, we infer that 
\begin{equation}\label{f6}
\sup_{\partial B_{R}} |u(x)| \rightarrow 0, \quad \text{as} \ R \rightarrow \infty. 
\end{equation}

Next, we prove the decay estimate for $ u(x)$. Take a positive constant $ R_0$ and
set $
M= \sup_{\partial B_{R_0}} |u(x)|$. By the $ C_{\text{loc}}^{1+\alpha}$ regularity estimates for $ p$-Laplace equation \cite{DiBenedetto}, we know that $ M $ is finite. For any $ \varepsilon > 0$, define
\[
\psi_{\varepsilon}(x):= M\left(\frac{R_0}{|x|}\right)^{\beta} + \varepsilon,
\]  
where
\[
\beta:= \frac{n-p}{p-1}.
\]
It is clear that 
\[
-\Delta_p \psi_{\varepsilon}(x) =0, \quad |x| > R_0. 
\]
By applying \eqref{f6}, we can choose a large constant $ R_1 > R_0 >0$ so that 
\[
|u(x)| \leq \varepsilon, \quad \text{for} \ x\in \partial B_{R_1}.
\]
Then $ u(x)$ and $ \psi_{\varepsilon}(x)$ satisfy
\begin{equation}
\begin{cases}
-\Delta_p u = -\Delta_p \psi_{\varepsilon} = 0 \quad & x \in B_{R_1}\setminus B_{R_0} , \\
u(x) \leq  \psi_{\varepsilon}(x) \qquad &  x \in \partial B_{R_0}, \\
u(x) \leq  \psi_{\varepsilon}(x) \qquad &  x \in \partial B_{R_1}. \nonumber 
\end{cases}
\end{equation}
By using the comparison principle (\cite[Theorem 2.15]{Notes}), one obtains
\[
u(x) \leq  \psi_{\varepsilon}(x), \quad x\in B_{R_1}\setminus B_{R_0}.
\]
Similarly, repeating the above process for $ -u(x)$, we deduce that
\[
|u(x)| \leq  \psi_{\varepsilon}(x), \quad x\in B_{R_1}\setminus B_{R_0}.
\]
Letting $ \varepsilon \rightarrow 0$, one has 
\[
|u(x)| \leq M R_0^{\beta} |x|^{-\beta} = C |x|^{-\beta}. 
\]
Thus, $ u(x) = O(|x|^{-\frac{n-p}{p-1}})$ as $|x| \to \infty$. We complete the proof.
\end{proof}

At the end of this section, we introduce a comparison principle on an exterior domain, by means of which we can prove that the solutions of \eqref{f9} have finite $ p$-energy. 

\begin{lemma}[Comparison principle]\label{f11}
Suppose that $ u$,$ v \in W^{1,p}_{\text{loc}}(\Omega^{\text{ext}}) $ are $ p$-harmonic functions in an exterior domain $ \Omega^{\text{ext}} \subset \mathbb{R}^n$. If for any $\xi \in \partial \Omega$
\[
\limsup_{x\rightarrow \xi, \, x\in \Omega^{\text{ext}} } (u(x) - v(x)) \leq 0
\]
and 
\[
\limsup_{|x| \rightarrow \infty, \, x\in \Omega^{\text{ext}}} (u(x) - v(x)) \leq 0,
\]
then $ u\leq v$ in $ \Omega^{\text{ext}}$. 
\end{lemma}

\begin{proof}
For any $ x_0 \in \Omega^{\text{ext}}$ and $ \varepsilon > 0$, there exists a $ R_\varepsilon$ such that 
\[
u(x) \leq v(x) + \varepsilon, \quad x\in \Omega^{\text{ext}}, \quad |x| \geq R_\varepsilon.
\]
We choose $ R > 0$ large enough so that 
\[
\overline{\Omega} \subset B_R, \qquad x_0 \in B_R\qquad  \text{and } \qquad R>R_\varepsilon.
\]
For the bounded domain $ \Omega^{\text{ext}}_R:=\Omega^{\text{ext}} \cap B_R$, on the inner boundary $ \partial \Omega \cap B_R$ and outer boundary $ \Omega^{\text{ext}} \cap \partial B_R$, one obtains 
\[
u(x) \leq v(x) + \varepsilon. 
\]
Since $ u(x)$ and $ v(x) + \varepsilon $ are $ p$-harmonic functions in $ \Omega^{\text{ext}}_R$, the comparison principle for bounded domains (\cite[Theorem 2.15]{Notes}) yields
\[
u(x) \leq v(x) + \varepsilon, \qquad \text{in} \ \Omega^{\text{ext}}_R.
\]
Letting $ \varepsilon \to 0$, since $ x_0 \in \Omega^{\text{ext}}$ is arbitrary, we have $ u\leq v$ in $ \Omega^{\text{ext}}$.
\end{proof}

Then, by applying the comparison principle and Proposition \ref{prop:decay-estimate}, we prove that any solution of the exterior Steklov problem has finite $p$-energy, that is, belongs to the space $ E^{1,p}(\Omega^{\text{ext}})$.

\begin{proposition}\label{prop:exactly-space}
Every solution of \eqref{f9} belongs to $ E^{1,p}(\Omega^{\text{ext}})$.
\end{proposition}

\begin{proof}
By the results of \cite{AH2014, Han2016}, there exists an eigenpair $ (\sigma_1, u_1) \in \mathbb{R}\times E^{1,p}(\Omega^{\text{ext}})$ satisfying
\begin{equation}
\begin{cases}
-\Delta_p u = 0 \quad &\text{ in } \ \Omega^{\text{ext}}, \\
|\nabla u|^{p-2} \frac{\partial u}{\partial \nu} = \sigma |u|^{p-2} u  &\text{ on } \  \partial\Omega.
\end{cases}
\end{equation}
Then, it is clear that $ u_1\in L^{p^*}(\Omega^{\text{ext}})$. According to the proof of Proposition~\ref{prop:decay-estimate}, we obtain that $ u_1(x) = O\left(|x|^{-\frac{n-p}{p-1}}\right)$ as $|x| \to \infty$. Therefore, the solution set of problem \eqref{f9} is nonempty. 

Then, we prove that all solutions of \eqref{f9} belong to $ E^{1,p}(\Omega^{\text{ext}})$. Let $ u(x)$ be a solution of \eqref{f9} and let $ f$ be the restriction of $ u$ to $ \partial \Omega$. It follows from \cite{AH2014} that there is a solution $ U \in E^{1,p}(\Omega^{\text{ext}})$ satisfying
\begin{equation}\label{f10}
\begin{cases}
-\Delta_p u = 0 \quad &\text{ in } \ \Omega^{\text{ext}}, \\
u = f  &\text{ on } \  \partial\Omega.
\end{cases}
\end{equation}
Similarly, we deduce that $ U(x) = O\left(|x|^{-\frac{n-p}{p-1}}\right)$ at $ \infty$. 

Since $ u(x) \in W^{1,p}_{\text{loc}}(\Omega^{\text{ext}})$ is a solution of \eqref{f9}, there exist positive constants $ C_1$, $ C_2$ and $ R_1$ such that 
\[
|u(x)| \leq C_1 |x|^{-\beta}, \quad |U(x)| \leq C_2 |x|^{-\beta}, \qquad |x| \geq R_1,
\]
where $ \beta:= \frac{n-p}{p-1}$. Thus, 
\[
|u(x) - U(x)| \leq (C_1 + C_2) |x|^{-\beta} \rightarrow 0, \quad \text{as} \quad |x| \rightarrow \infty. 
\]
Since $ u(x) = U(x)$ on $ \partial \Omega$, by using Lemma \ref{f11}, we have $ u(x) = U(x)$. Then, $ u(x) \in E^{1,p}(\Omega^{\text{ext}})$.
\end{proof}

\section{Estimate of exterior Steklov eigenvalue}\label{sec:eigenvalue}
In this section, we study the exterior Steklov problem and obtain upper and lower estimates for the first variational eigenvalue $ \sigma_{1,p}$. Inspired by the proof of Theorem 1.11 in \cite{LADMI2026}, we provide the following lower bound for the first variational eigenvalue of the exterior Steklov problem for the $ p$-Laplace equation. That is, for any $1 < p <n$,
\begin{equation}
\sigma_{1,p}(\Omega^{\text{ext}}) \geq \inf_{s\in \partial \Omega} \left[\frac{n-p}{p-1}\mathcal{L}_p(\kappa_1(s), \cdots, \kappa_{n-1}(s))\right]^{p-1},
\end{equation}
where $ \kappa_j(s)$ is the $ j$-th principal curvature of $ \partial \Omega$.

\begin{proof}[Proof of Theorem~\ref{f8}] 
Since $\Omega$ is a convex domain, we know that $ \prod_{j=1}^{n-1} \kappa_j(s) \geq 0$. Assume that there exist $ q$ of the $ \{\kappa_j(s) \}_{j=1}^{n-1}$ that are positive. If $ q \leq p-1$, then by convention, the logarithmic mean $ \mathcal{L}_p(\kappa_1(s), \cdots, \kappa_{n-1}(s)) =0$. Thus, in the following we consider $ p-1 < q \leq n-1$. 

From the definition of eigenvalues, we know that
\[
 \sigma_{1,p}(\Omega^{\text{ext}}) = \inf_{\substack{u \in E^{1,p}(\Omega^{\text{ext}}) \\ u \neq 0}} \frac{\displaystyle\int_{\Omega^{\text{ext}}} |\nabla u|^p \, \mathrm{d}x}{\displaystyle\int_{\partial \Omega} |u|^p \, \mathrm{d}S}.
\]
Inspired by the method in \cite[Theorem 1.1]{KP2017}, we consider the map 
\[
\Psi: \partial \Omega \times (0,\infty) \rightarrow \Omega^{\text{ext}}, \qquad (s,t) \mapsto s -t \nu(s).
\]
$ \Psi $ is bijective and locally bi-Lipschitz. $\nu(s) $ is the unit normal pointing towards the exterior of $\Omega^{\text{ext}}$.  Any $ u\in E^{1,p} $ can be approximated by functions with compact support in $ \Omega^{\text{ext}}$.   
A change of variables yields
\[
\int_{\Omega^{\text{ext}}} |\nabla u|^p \, \mathrm{d}x = \int_{\partial \Omega \times (0,\infty)} |\nabla u( \Psi (s,t))|^p \det (D\Psi)\, \mathrm{d}S\mathrm{d}t,
\]
where the Jacobian determinant is given by 
\[
\zeta(s, t) := \det(D\Psi) = \prod_{j=1}^{n-1} \left(1+ \kappa_j(s)t \right).
\]
Define $ \omega(s,t) := u(\Psi(s,t))$. By the chain rule, 
\[
\partial_t \omega(s,t) = - \langle \nabla u( \Psi (s,t)), \nu(s) \rangle
\] 
and then 
\[
|\nabla u( \Psi (s,t))| \geq | \langle \nabla u( \Psi (s,t)), \nu(s) \rangle| = |\partial_t \omega(s,t)|.
\] 
It is clear that
\[
\int_{\Omega^{\text{ext}}} |\nabla u|^p \, \mathrm{d}x \geq \int_{\partial \Omega \times (0,\infty)} |\partial_t \omega(s,t)|^p \zeta(s, t) \, \mathrm{d}S\mathrm{d}t.
\]
Hence, we obtain 
\[
\sigma_{1,p}(\Omega^{\text{ext}}) \geq \inf_{0\neq u \in E^{1,p}(\Omega^{\text{ext}})} \frac{\displaystyle\int_{\partial \Omega \times (0,\infty)} |\partial_t \omega(s,t)|^p \zeta(s, t)\, \mathrm{d}S\mathrm{d}t}{\displaystyle\int_{\partial \Omega} |\omega(s,0)|^p \, \mathrm{d}S}. 
\]

For every $s \in \partial \Omega$, define
\begin{equation}\label{equ:definition_of_K}
 K_p(s, \Omega^{\text{ext}}):= \inf_{\substack{0\neq f \in W_{\text{loc}}^{1, p}((0, \infty)) \\ \lim\limits_{x \to \infty}f(x) = 0}}\frac{\displaystyle\int_0^{\infty}|f'(t)|^p \zeta(s, t)\, \mathrm{d}t}{|f(0)|^p},
\end{equation}
which gives
\[
\sigma_{1,p}(\Omega^{\text{ext}}) \geq \inf_{0\neq u \in E^{1,p}(\Omega^{\text{ext}})} \frac{ \displaystyle\int_{\partial \Omega }K_p(s, \Omega^{\text{ext}})  |\omega(s,0)|^p \, \mathrm{d}S}{\displaystyle\int_{\partial \Omega} |\omega(s,0)|^p \, \mathrm{d}S} \geq \inf_{s\in \partial \Omega} K_p(s, \Omega^{\text{ext}}). 
\]

For \eqref{equ:definition_of_K}, we consider a minimizing sequence $\{f_k\}_{k \in \mathbb{N}} \subset W^{1, p}_{\text{loc}}((0, \infty))$ with $f_k(0) = 1$, $f_k(x) \to 0$ as $x \to \infty$ and $\lim\limits_{k \to \infty}\displaystyle\int_0^{\infty}|f_k'(t)|^p \zeta(s, t) \, \mathrm{d}t = K_p(s, \Omega^{\text{ext}})$. For each $R > 0$, $\zeta(s, t) \geq 1$ imply that the sequence is bounded in $W^{1, p}((0, R))$. Hence, by standard compactness results, there exists a subsequence, again denoted by $f_k$, that converges weakly in $W^{1, p}((0, R))$ to a function $f \in W^{1, p}((0, R))$. 

We then show that $f$ decays at infinity. Since $ q$ of the principle curvatures are positive, there exists a constant $D_1(s) > 0$ such that $\zeta(s, t) \geq D_1(s)t^{q}$. Thus, for any $x \in (0, \infty)$,
\begin{align*}
 \int_x^{\infty}\left(\frac{1}{\zeta(s, t)}\right)^{\frac{1}{p-1}} \, \mathrm{d}t &\leq \left(\frac{1}{D_1(s)}\right)^{\frac{1}{p-1}}\int_x^{\infty}t^{-\frac{n-1}{p-1}} \, \mathrm{d}t \\
 &= \left(\frac{1}{D_1(s)}\right)^{\frac{1}{p-1}} \frac{p-1}{q-(p-1)}x^{\frac{p-1-q}{p-1}}.
\end{align*}
It follows that there exists a constant $D_2(s, p, q) > 0$ such that for any $x \in (0, \infty)$, 
\begin{align*}
 |f_k(x)| &\leq \Big|\int_x^{\infty}f'_k(t) \, \mathrm{d}t \Big| \leq \int_x^{\infty}|f'_k(t)|\zeta^{\frac{1}{p}}(s, t)\zeta^{-\frac{1}{p}}(s, t) \, \mathrm{d}t \\
 &\leq\left(\int_x^{\infty}|f'_k(t)|^p\zeta(s, t) \, \mathrm{d}t\right)^{\frac{1}{p}}\left(\int_x^{\infty}\left(\zeta^{-\frac{1}{p}}(s, t)\right)^{\frac{p}{p-1}} \, \mathrm{d}t\right)^{\frac{p-1}{p}}\\
 &\leq D_2(s, n, p)x^{\frac{p-1-q}{p}}.
\end{align*}
The condition $ q>p-1$ and the uniform decay imply that $f(x) \to 0$ as $x \to \infty$. Since $f$ is not a constant, we have $K_p(s, \Omega^{\text{ext}}) > 0$. Moreover, $f$ satisfies the following Euler-Lagrange equation
\begin{equation} \label{ELequa}
 \begin{cases}
  (|f'(t)|^{p-2} f'(t) \zeta (s,t) )' = 0,  \quad \text{ for }  t\in (0, \infty), \\
  -|f'(0)|^{p-2} f'(0) =  |f(0)|^{p-2} f(0) K_p(s,\Omega^{\text{ext}}), \\
  f(t) \rightarrow 0$ \text{ as } $ t\to \infty.
 \end{cases}
\end{equation}

Define a function $\Phi: \mathbb{R} \to \mathbb{R}$ as 
\[
 \Phi(\xi) = |\xi|^{p-2}\xi.
\]
Clearly, 
\[
 \Phi^{-1}(y) = \mathrm{sgn}(y)|y|^{\frac{1}{p-1}}.
\]
Then, $(|f'(t)|^{p-2} f'(t) \zeta (s,t) )' = 0$ implies that there exists $C_1(s)$ such that 
\[
 |f'(t)|^{p-2}f'(t) = \frac{C_1(s)}{\zeta(s, t)},
\]
and thus,
\begin{align*}
 f'(t) &= \mathrm{sgn}\left(\frac{C_1(s)}{\zeta(s, t)}\right)\Big|\frac{C_1(s)}{\zeta(s, t)}\Big|^{\frac{1}{p-1}} = \mathrm{sgn}(C_1(s))|C_1(s)|^{\frac{1}{p-1}}\zeta^{-\frac{1}{p-1}}(s, t)\\
 &=C_2(s)\zeta^{-\frac{1}{p-1}},
\end{align*}
with $C_2(s) = \mathrm{sgn}(C_1(s))|C_1(s)|^{\frac{1}{p-1}}$. It follows that
\begin{equation}\label{n1}
 f(t) = C_3(s) + C_2(s)\int_0^t\zeta^{-\frac{1}{p-1}}(s, r) \, \mathrm{d}r.
\end{equation}
The condition $\lim\limits_{t \to \infty}f(t) = 0$ implies that 
\[
 \frac{C_2(s)}{C_3(s)} = -\frac{1}{\displaystyle \int_0^{\infty}\zeta^{-\frac{1}{p-1}}(s, r) \, \mathrm{d}r} = -\frac{1}{I_p(s)}, \qquad I_p(s) := \int_0^{\infty}\zeta^{-\frac{1}{p-1}}(s, t) \, \mathrm{d}t. 
\]
Hence, by \eqref{ELequa}-\eqref{n1}
\[
 K_p(s, \Omega^{\text{ext}}) = -\frac{|f'(0)|^{p-2}f'(0)}{|f(0)|^{p-2}f(0)} = -\frac{|C_2(s)|^{p-2}C_2(s)}{|C_3(s)|^{p-2}C_3(s)} = I_p(s)^{1-p}.
\]

We claim that
\begin{equation}\label{equ:Dirichlet_equality}
 \zeta(s, t)^{-\frac{1}{p-1}} = \frac{\Gamma\left(\frac{n-1}{p-1}\right)}{\Gamma\left(\frac{1}{p-1}\right)^{n-1}}\int_{S_{n-2}}\frac{\prod_{j=1}^{n-1}u_j^{\frac{1}{p-1}-1}}{(1 + t\sum_{j=1}^{n-1}\kappa_j(s)u_j)^{\frac{n-1}{p-1}}} \, \mathrm{d}u_1\cdots \mathrm{d}u_{n-2},
\end{equation}
where $S_{n-2} = \{u_j \geq 0: \sum_{j=1}^{n-1}u_j = 1\}$. 

Recall that 
\[
 \zeta(s, t) = \prod_{j = 1}^{n-1}(1 + \kappa_j(s)t).
\]
Indeed, for any $a, b > 0$, since
\[
 \int_0^{\infty}x^{a-1}e^{-bx} \, \mathrm{d}x = \frac{\Gamma(a)}{b^a},
\]
we have
\[
 (1 + \kappa_j(s)t)^{-\frac{1}{p-1}} = \frac{1}{\Gamma\left(\frac{1}{p-1}\right)}\int_0^{\infty}x_j^{\frac{1}{p-1}-1}e^{-(1 + \kappa_j(s)t)x_j} \, \mathrm{d}x_j.
\]
It follows that 
\begin{align*}
  &\mathrel{\phantom{=}}\prod_{j=1}^{n-1}(1 + \kappa_j(s)t)^{-\frac{1}{p-1}}\\
  &= \frac{1}{\Gamma\left(\frac{1}{p-1}\right)^{n-1}}\int_0^{\infty}\cdots\int_0^{\infty}\left(\prod_{j=1}^{n-1}x_j^{\frac{1}{p-1}-1}\right)e^{-\sum_{j=1}^{n-1}(1 + \kappa_j(s)t)x_j} \, \mathrm{d}x_1\cdots \mathrm{d}x_{n-1} \\
  &= \frac{1}{\Gamma\left(\frac{1}{p-1}\right)^{n-1}}\int_0^{\infty}\cdots\int_0^{\infty}\left(\prod_{j=1}^{n-1}x_j^{\frac{1}{p-1}-1}\right)e^{-\sum_{j=1}^{n-1}x_j}e^{-t\sum_{j=1}^{n-1}\kappa_j(s)x_j} \, \mathrm{d}x_1\cdots \mathrm{d}x_{n-1}.
\end{align*}
Let $T = \sum_{j = 1}^{n-1}x_j$ and $u_j = \frac{x_j}{T}$. Then 
\[
 \mathrm{d}x_1\cdots \mathrm{d}x_{n-1} = T^{n-2}\, \mathrm{d}T\mathrm{d}u_1\cdots\mathrm{d}u_{n-2}.
\]
Thus, 
\begin{align*}
 &\mathrel{\phantom{=}}\prod_{j=1}^{n-1}(1 + \kappa_j(s)t)^{-\frac{1}{p-1}}\\
 &=\frac{1}{\Gamma\left(\frac{1}{p-1}\right)^{n-1}}\int_{S_{n-2}}\left(\prod_{j=1}^{n-1}u_j^{\frac{1}{p-1}-1}\right)\int_0^{\infty}T^{\frac{n-1}{p-1}-1}e^{-T}e^{-tT\sum_{j=1}^{n-1}\kappa_j(s)u_j} \, \mathrm{d}T\mathrm{d}u_1\cdots\mathrm{d}u_{n-2} \\
 &= \frac{1}{\Gamma\left(\frac{1}{p-1}\right)^{n-1}}\int_{S_{n-2}}\left(\prod_{j=1}^{n-1}u_j^{\frac{1}{p-1}-1}\right)\frac{\Gamma\left(\frac{n-1}{p-1}\right)}{(1 + t\sum_{j=1}^{n-1}\kappa_j(s)u_j)^{\frac{n-1}{p-1}}} \, \mathrm{d}u_1\cdots\mathrm{d}u_{n-2},
\end{align*}
i.e., \eqref{equ:Dirichlet_equality} holds. Integrating both sides of \eqref{equ:Dirichlet_equality} with respect to $t$ over $ (0, +\infty)$, we obtain
\[
 I_p(s) = \int_0^{\infty}\zeta^{-\frac{1}{p-1}}(s, t) \, \mathrm{d}t = \frac{1}{\frac{n-p}{p-1}\mathcal{L}_p(\kappa_1(s), \cdots, \kappa_{n-1}(s))},
\]
where the definition of logarithmic mean
\[
 \mathcal{L}_p(\kappa_1(s), \cdots, \kappa_{n-1}(s)):= \left[\frac{\Gamma\left(\frac{n-1}{p-1}\right)}{\Gamma\left(\frac{1}{p-1}\right)^{n-1}}\int_{S_{n-2}}\frac{\prod_{j=1}^{n-1}u_j^{\frac{1}{p-1}-1}}{\sum_{j=1}^{n-1}\kappa_j(s)u_j} \, \mathrm{d}u_1\cdots \mathrm{d}u_{n-2}\right]^{-1}.
\]

Thus, 
\[
 K_p(s, \Omega^{\text{ext}}) = \left[\frac{n-p}{p-1}\mathcal{L}_p(\kappa_1(s), \cdots, \kappa_{n-1}(s))\right]^{p-1}.
\]
Since $ \sigma_{1,p}(\Omega^{\text{ext}}) \geq \inf_{s\in \partial \Omega} K_p(s, \Omega^{\text{ext}})$, we get the desired result. 
\end{proof}

Then, we prove an upper bound estimate for the first variational eigenvalue $ \sigma_{1, p}(\Omega^{\text{ext}})$.

\begin{proof}[Proof of Theorem~\ref{thm:main_upper}]
 For any $\varepsilon > 0$, there exists $A \subset \partial \Omega$ such that 
 \[
  \frac{\mathrm{Cap}_p(A; \partial \Omega)}{\mathrm{Area}(A)} \leq \Lambda_p(\partial \Omega) + \varepsilon.
 \]
 By the definition of capacity, there exists a function $g$ satisfying $g \in C^1(\Omega^{\text{ext}}) \cap C(\overline{\Omega^{\text{ext}}})$, $\mathrm{supp}(g)$ is bounded and $g|_A = 1$ such that 
 \[
  \int_{\Omega^{\text{ext}}}|\nabla g|^p \, \mathrm{d}x \leq \mathrm{Cap}_p(A; \partial \Omega) + \varepsilon \, \mathrm{Area}(A).
 \]
 Clearly, $g \in E^{1, p}(\Omega^{\text{ext}})$. Then we have
 \begin{align*}
  \sigma_{1, p}(\Omega^{\text{ext}}) &\leq \frac{\int_{\Omega^{\text{ext}}}|\nabla g|^p \, \mathrm{d}x}{\int_{\partial \Omega}|g|^p \, \mathrm{d}S} \leq \frac{\mathrm{Cap}_p(A; \partial \Omega) + \varepsilon \, \mathrm{Area}(A)}{\int_{\partial \Omega}|g|^p \, \mathrm{d}S} \\
  &\leq \frac{\mathrm{Cap}_p(A; \partial \Omega) + \varepsilon \, \mathrm{Area}(A)}{\mathrm{Area}(A)} \leq \Lambda_p(\partial \Omega) + 2\varepsilon.
 \end{align*}
Letting $\varepsilon \to 0$ yields
 \[
  \sigma_{1, p}(\Omega^{\text{ext}}) \leq \Lambda_p(\partial \Omega).
 \]
 This completes the proof.
\end{proof}

By virtue of the results on the upper and lower bounds for the first variational eigenvalue $ \sigma_{1, p}$, we obtain the exact value of the first eigenvalues for $ B_R^{\text{ext}}$. This shows that the estimates in Theorem \ref{f8} and Theorem \ref{thm:main_upper} are sharp.

\begin{example}\label{exm:sharp}
 Taking $t = 0$ in \eqref{equ:Dirichlet_equality}, we have
 \[
  \frac{\Gamma\left(\frac{n-1}{p-1}\right)}{\Gamma\left(\frac{1}{p-1}\right)^{n-1}}\int_{S_{n-2}}\prod_{j=1}^{n-1}u_j^{\frac{1}{p-1}-1} \, \mathrm{d}u_1\cdots \mathrm{d}u_{n-2} = 1.
 \]
 If $\Omega = B_R$ for some $R > 0$, then for any $s \in \partial B_R$ and $j \in \{1, \cdots, n-1\}$, $\kappa_j(s) = \frac{1}{R}$. It follows that 
 \[
  \mathcal{L}_p(\kappa_1(s), \cdots, \kappa_{n-1}(s)) = \frac{1}{R}.
 \]
 By using Theorem \ref{f8}, one obtains 
 \[
 \sigma_{1, p}(B_R^{\text{ext}}) \geq  K_p(s, B_R^{\text{ext}}) = \left(\frac{n-p}{R(p-1)}\right)^{p-1}.
 \]
 On the other hand, by classical capacity theory, 
 \[
  \mathrm{Cap}_p(\partial B_R; \partial B_R) = \omega_n \left(\frac{n-p}{p-1}\right)^{p-1}R^{n-p},
 \]
 where $\omega_n$ is the $(n-1)$-dimensional measure of $\mathbb{S}^{n-1}$; see e.g., \cite[Section 2.2.4]{Mazya2011}. Then by Theorem~\ref{thm:main_upper}, we have
 \[
  \sigma_{1, p}(B_R^{\text{ext}}) \leq \frac{\mathrm{Cap}_p(\partial B_R; \partial B_R)}{\mathrm{Area}(\partial B_R)} = \frac{\omega_n \left(\frac{n-p}{p-1}\right)^{p-1}R^{n-p}}{\omega_n R^{n-1}} = \left(\frac{n-p}{R(p-1)}\right)^{p-1}.
 \]
 Hence,
 \[
  \sigma_{1, p}(B_R^{\text{ext}}) = \left(\frac{n-p}{R(p-1)}\right)^{p-1}.
 \]
\end{example}

In particular, for the case $ R=1$ (which yields $ \sigma_{1, p}(B_1^{\text{ext}}) = \left(\frac{n-p}{p-1}\right)^{p-1}$), Han provided the relevant results in \cite{Han2016} and derived the corresponding eigenfunction $ u_1 = |x|^{\frac{p-n}{p-1}}.$ Correspondingly, the eigenfunction is given by $ u_R= \left(\frac{R}{|x|}\right)^{\frac{n-p}{p-1}}$ for the domain $ B_R^{\text{ext}}$. 
 
\noindent \textbf{Acknowledgments.} \ The first author is supported by NSFC of China (No.12501266), Natural Science Foundation of Jiangsu Province (No.BK20250703). The second author is partially supported by the National Key R\&D Program of China (No.2025YFA1017500). The authors sincerely appreciate the editors and referees for their careful reading and helpful comments to improve this paper.

\bibliographystyle{plain}
\bibliography{plap}

\end{document}